\documentclass[12pt,a4paper,psamsfonts]{amsart}
\usepackage[english]{babel}
\usepackage[T1]{fontenc}
\usepackage{mathtools}
\usepackage{fancyhdr}
\usepackage{appendix}
\usepackage{amssymb,amscd,amsxtra,calc}
\usepackage{mathrsfs}
\usepackage{xcolor}
\newcommand{\revision}[1]{\leavevmode{\color{blue}#1}}
\usepackage{cmmib57}
\usepackage{multirow}
\usepackage[all]{xy}
\usepackage{longtable}
\usepackage[colorlinks=true,linkcolor=blue,anchorcolor=blue,citecolor=blue,pagebackref,linktocpage]{hyperref}
\usepackage{cleveref}
\usepackage{tikz}
\usepackage{cite}
\theoremstyle{plain}
    \newtheorem{thm}{Theorem}[section]

     \newtheorem{conjecture}[thm]{Conjecture}
    \newtheorem{corollary}[thm]{Corollary}
    
    \newtheorem{lemma}[thm]{Lemma}
    \newtheorem{proposition}[thm]{Proposition}
    
    \newtheorem{theorem}[thm]{Theorem}
    \newtheorem{problem}[thm]{Problem}

\theoremstyle{definition}
    \newtheorem{definition}[thm]{Definition}

    \newtheorem*{notation*}{Notation and Terminology}
      
    \newtheorem{remark}[thm]{Remark}
    
\theoremstyle{remark}

\newcommand{\arxiv}[1]{\href{https://arxiv.org/abs/#1}{{\tt arXiv:#1}}}
\newcommand{\C}{\mathbb{C}}

\newcommand{\bP}{\mathbb{P}}

\newcommand{\bR}{\mathbb{R}}
\newcommand{\bC}{\mathbb{C}}

\newcommand{\OO}{\mathcal{O}}

\newcommand{\NE}{\overline{\operatorname{NE}}}
\newcommand{\Nef}{\operatorname{Nef}}

\newcommand{\Null}{\operatorname{Null}}

\newcommand{\Supp}{\operatorname{Supp}}

\newcommand{\Z}{\mathbb Z}
\newcommand{\R}{\mathbb R}
\newcommand{\PP}{\mathbb P}
\newcommand{\cC}{\mathcal C}
\newcommand{\cT}{\mathcal T}
\newcommand{\MW}{\operatorname{MW}}

\newcommand{\lcm}{\operatorname{lcm}}

\newcommand{\Pic}{\operatorname{Pic}}

\newcommand{\mstriangle}[1]{
\begin{tikzpicture}[x=0.3cm,y=0.3cm]
\draw (-0.4,-0.433) -- (1.4,-0.433);
\draw (-0.2,-0.7794) -- (0.7,0.7794);
\draw (1.2,-0.7794) -- (0.3,0.7794);
\end{tikzpicture}
}
\newcommand{\mssharp}[1]{
\begin{tikzpicture}[x=0.3cm,y=0.3cm]
\draw (-0.8,-0.5) -- (0.8,-0.5);
\draw (-0.8,0.5) -- (0.8,0.5);
\draw (-0.5,-0.8) -- (-0.5,0.8);
\draw (0.5,-0.8) -- (0.5,0.8);
\end{tikzpicture}
}

\makeatletter

\newcommand{\Rmnum}[1]{\expandafter\@slowromancap\romannumeral #1@}
\makeatother

\title[Jacobian elliptic surfaces of Kodaira dimension one]
{Semiampleness on Jacobian elliptic surfaces of Kodaira dimension one}
\author{Antonio Laface}
\address{Departamento de Matem\'atica, Universidad de Concepci\'on, Casilla 160-C, Concepci\'on, Chile}
\email{\href{mailto:alaface@udec.cl}{alaface@udec.cl}}

\author{Sichen Li}
\address{
School of Mathematics, East China University of Science and Technology, Shanghai 200237, P. R. China}
\email{\href{mailto:sichenli@ecust.edu.cn}{sichenli@ecust.edu.cn}}


\begin{document}

\begin{abstract}
Let $\pi: X\to \bP^1$ be a semistable Jacobian elliptic surface over $\bC$, and set $\chi=\chi(\mathcal O_X)\ge3$, so that $\kappa(X)=1$.
Assume that the Mordell-Weil group of $\pi$ is finite and that $\pi$ has at least one reducible fiber, the reducible fibers being of types $I_{n_1},\cdots, I_{n_s}$.
Recently, Laface et al. proved that the zero section and the components of the reducible fibers generate $\NE(X)$ if and only if
$$\delta(\pi):=\sum_{i=1}^s\frac{\lfloor n_i^2/4\rfloor}{n_i}\le\chi.$$
In particular, the Mori cone is rational polyhedral in this range.
They also proved that
 $N(\pi)=\sum_{i=1}^s n_i\le 2\chi+3$ implies that $X$ is a Mori dream surface.
In this paper, we study the existence problem of Mori dream surfaces provided that $N(\pi)\ge 2\chi+4$.
Suppose $\delta(\pi)\le \chi$.
We first show that every nef isotropic divisor on $X$ is semiample whenever each $n_i$ is even.
Furthermore,  when every $n_i$ is even, we obtain criteria for $X$ to be a Mori dream surface: $X$ is  a Mori dream surface provided that either (i) all $n_i=2$, or (ii) $N(\pi)\le 2\chi+4$; if some $n_j=4$, then $X$ is a Mori dream surface if and only if $N(\pi)\le 2\chi+4$.
However, once some $n_i$ is odd,  we construct a Jacobian elliptic surface \(\pi: Y\to\mathbb P^1\) with \(\delta(\pi)=\chi=3\), trivial Mordell-Weil group, and singular-fiber configuration
\(I_4+3I_3+23I_1\) for which none of the eight non-vertical isotropic extremal rays is semiample. In particular, $Y$ is not a Mori dream surface.
\end{abstract}

\maketitle
\section{Introduction}
Mori dream spaces were introduced by Hu and Keel \cite{HK00} as varieties on
which the main steps of the minimal model program can be organized into
finitely many birational models. Equivalently, under the standard hypotheses,
their Cox rings are finitely generated. For a smooth projective surface $X$
with $q(X)=0$, the criterion takes a particularly concrete form: $X$ is a
Mori dream surface if and only if its effective cone is rational polyhedral
and every nef divisor is semiample; see \cite[Corollary~2.6]{AHL10}.
Over the last dozen years, substantial progress has been made toward classifying Mori dream surfaces $X$ with $\kappa(X)\ne1$  \cite{AHL10, AL11, KL19, TVV11}, also see  \cite[Chapter 5]{ADHL15} and the references therein for more detail.

For Jacobian elliptic surfaces of Kodaira dimension one, rational
polyhedrality of the Mori cone leaves two distinct semiampleness
questions, concerning isotropic and big nef rays. We study both
questions for semistable fibrations with finite Mordell--Weil group.
Assuming $\delta(\pi)\leq\chi$, when every reducible fiber is even we prove that all nef isotropic
divisors are semiample and obtain criteria for the Mori dream property
by studying the big nef rays. An example with odd fibers shows that
the isotropic conclusion can fail. The semiampleness result also yields
the bounded cohomology property, including cases in which the surface
is not a Mori dream surface.

Let $\pi\colon X\to\PP^1$ be a relatively minimal semistable Jacobian elliptic surface over $\C$ with finite Mordell--Weil group, and we write $\chi=\chi(\OO_X)\geq3$.  We denote by $F$ the class of a fiber and by $C_0$ the zero section.  Thus $q(X)=0$, $K_X\sim(\chi-2)F$, $C_0^2=-\chi$, and $\kappa(X)=1$.
Suppose that the reducible fibers are $F_i=C_{i,0}+\cdots+C_{i,n_i-1}$, for $1\leq i\leq s$, of types $I_{n_i}$.  The components are cyclically ordered and $C_0$ meets $C_{i,0}$.  If $\cT$ denotes the set of all components of the reducible fibers, set the fibration cone
\[
   \cC_\pi=\R_{\geq0}[C_0]+\sum_{T\in\cT}\R_{\geq0}[T],
\]
and write 
$$
\delta(\pi):=\sum_{i=1}^s\frac{\lfloor n_i^2/4\rfloor}{n_i}, \qquad N(\pi):=\sum_{i=1}^s n_i.
$$

Recently, Laface et al.  \cite[Theorem 1.1]{LLL26} proved the following theorem.
\begin{theorem}
\label{NE-thm}
Let $\pi: X\to \bP^1$ be a Jacobian elliptic surface with $\chi:=\chi(\mathcal O_X)\ge3$ and finite Mordell-Weil group.
Suppose that $\pi$ has at least one reducible fiber and that every reducible fiber is of type $I_n$.
Then the following statements hold.
\begin{enumerate}
	\item The equality $\mathcal C_\pi=\NE(X)$ holds if and only if $\delta(\pi)\le \chi$.
	In particular, the Mori cone is rational polyhedral in this range.
	\item If   $N(\pi)\le 2\chi+3$, then $X$ is a Mori dream surface.
\end{enumerate}
\end{theorem}
In the polyhedral range, Laface et al. \cite[Theorem 1.5]{LLL26}  constructed  a Jacobian elliptic surface   with $(\chi, n)=(3,11)$, singular-fiber configuration $I_{11}+25I_1$ and a big and nef divisor which is not semiample, where $N(\pi)=3\chi+2$.
However, it remains completely open when $X$ is a Mori dream surface provided that \(N(\pi)\ge 2\chi+4\) and \(\delta(\pi)\le \chi\).
\begin{problem}
\label{Prob-Main}
In the setting of Theorem \ref{NE-thm}, assume that  \(N(\pi)\ge 2\chi+4\) and \(\delta(\pi)\le \chi\).
When is $X$ a Mori dream surface?
\end{problem}
In the range $\delta(\pi)\leq\chi$, nonvertical isotropic nef rays occur
only on the boundary $\delta(\pi)=\chi$. This leads to the following
problem from our previous work.
\begin{problem}
\label{prob:isotropic}
\cite[Problem 1.3]{LLL26}
Assume that $\MW(\pi)$ is finite and $\delta(\pi)=\chi$.  
Let $D$ span a nonvertical extremal ray of $\Nef(X)$ and suppose that $D^2=0$.  Is $\kappa(X,D)>0$, or equivalently, is $D$ semiample?
\end{problem}
Below is our first main result in this paper.
\begin{theorem}
\label{Mainthm-even}
Let $\pi: X\to \bP^1$ be a Jacobian elliptic surface with $\chi:=\chi(\mathcal O_X)\ge3$ and finite Mordell-Weil group.
Suppose $\delta(\pi)\le \chi$ and the reducible fibers have types $I_{2m_1},\ldots,I_{2m_s}$.
Put $T=\sum_{i=1}^s m_i$.
Then the following statements hold.
\begin{enumerate}
\item[(a)] In the setting of Problem \ref{prob:isotropic}, every such $D$ is semiample.
As a result, $X$ satisfies the bounded cohomology property (cf. Definition \ref{BCP-defn}).
\item[(b)] If each $m_i$ is equal to 1, then $X$ is a Mori dream surface, where $N(\pi)\le4\chi$ holds.
\item[(c)] If $N(\pi)\leq2\chi+4$, then $X$ is a Mori dream surface.
\item[(d)] If some $m_i\geq2$ satisfies $T-m_i>\chi$, then $X$ has a
big and nef divisor which is not semiample.  In particular, $X$ is
not a Mori dream surface.
\item[(e)] If at least one reducible fiber is of type $I_4$, then $X$ is
a Mori dream surface if and only if $N(\pi)\leq2\chi+4$.
\end{enumerate}
\end{theorem}
\begin{remark}
The assumption that an $I_4$ fiber occurs is essential to the
classification statement proved here.  If all reducible fibers
are $I_2$, Theorem~\ref{thm:all-I2-mds} gives the Mori dream property
throughout $T\leq2\chi$.  If there is no $I_4$ fiber but larger even
fibers occur, Theorem \ref{Mainthm-even} (c) and (d)
 give sufficient conditions in opposite
directions, respectively; they do not settle every configuration.
\end{remark}
\begin{remark}
See Theorem \ref{thm:example-8I_2} for an example of  Mori dream Jacobian elliptic surfaces 
$
\pi:X\to  \PP^1
$
with $\delta(\pi)=\chi=4,\rho(X)=10$, and singular-fiber configuration $8I_2+32I_1$.
As an application of Theorem \ref{Mainthm-even}, we provide a new evidence for surfaces satisfying the bounded cohomology property, see Section \ref{Sect-BCP}.
Theorem \ref{Mainthm-even} (c) shows that $X$ is a Mori dream surface 
when $\chi=3$ and $\pi$ has exactly one reducible fiber of type $I_{10}$.
Thus $N(\pi)$ in  \cite[Theorem 1.5]{LLL26} is sharp for \(\chi=3\) when \(\pi\) has exactly one reducible fiber.
\end{remark}
However,  when \(\pi\) admits at least one reducible fiber of type \(I_{2m+1}\), we construct an example of semistable Jacobian elliptic surfaces $X$ which gives a complete negative answer to Problem \ref{prob:isotropic}, but $X$ still satisfies the bounded cohomology property.
\begin{theorem}
\label{thm-odd}
There exists a semistable Jacobian elliptic surface
$\pi\colon X\to\PP^1$ with $\delta(\pi)=\chi=3$, trivial Mordell--Weil group, and
singular-fiber configuration
\[
I_4+3I_3+23I_1
\]
such that the following statements hold.
\begin{enumerate}
\item None of the eight nonvertical isotropic extremal rays of
$\Nef(X)$ is semiample.
More precisely, if $D$ is the primitive integral generator of any of these
rays, then $$h^0(X,\OO_X(D))=1,\quad  \kappa(X,D)=0, \quad 
\pi_*\OO_X(D)\simeq
\OO_{\PP^1}\oplus\OO_{\PP^1}(-2)
\oplus\OO_{\PP^1}(-1)^{\oplus4}.$$
\item $X$  satisfies the bounded cohomology property.
More precisely, for every irreducible and reduced curve $C$ on $X$, $h^1(\mathcal O_X(C))\le h^0(\mathcal O_X(C))$.
\end{enumerate}
\end{theorem}
\begin{remark}
The even-fiber hypothesis in Theorem \ref{thm:all-even} is therefore
essential.  Already for the first numerically possible configuration
containing odd reducible fibers, the balanced splitting above occurs
and the obstruction
$\eta_S=[\OO_{D_S}(D_S)]$ is nonzero for every isotropic selection.
\end{remark}
For a deeper study of Problem \ref{Prob-Main} in the case where at least one odd fiber occurs, Theorem \ref{thm-odd} and Proposition \ref{prop:odd-example-bcp} motivate us to propose the following problem.
\begin{problem}
In the setting of Theorem \ref{NE-thm}, assume that one odd fiber occurs, \(\delta(\pi)\le \chi\) and \(N(\pi)\ge 2\chi+4\). 
\begin{enumerate}
	\item Could $X$ be a Mori dream surface? In general, classify Mori dream semistable Jacobian ellptic surfaces with at least one odd fiber.
	\item Could $X$ still satisfy the bounded cohomology property provided that $\delta(\pi)=\chi\ge4$ and there exists a nef isotropic divisor $D$ with $\kappa(X,D)=0$?
\end{enumerate}
\end{problem}
The paper is organized as follows.
Section \ref{Pre} records elliptic-fibration, vector-bundle tools and bounded cohomology property  used latter.
We prove Theorem \ref{Mainthm-even} and \ref{thm-odd}  in Section \ref{Sect-alleven} and  \ref{sec:odd-counterexample} respectively.
\subsection*{Acknowledgments}
Antonio Laface was partially supported by Proyecto FONDECYT Regular n. 1230287.
Sichen Li would like to thank Rong Du for constant encouragement.
\section{Nef isotropic divisors to be semiample}
\label{Pre}
\subsection{Notation from the fibration cone}

We use the notation of \cite[Subsections~4.1-4.2]{LLL26}.  
A \emph{selection} is a set $S\subset\cT$ containing one component $T_i$ of each reducible fiber.  
Let $d(T_i)$ be the cyclic distance from $C_{i,0}$ to $T_i$, chosen in $[0,\lfloor n_i/2\rfloor]$, and put $$\delta(T_i)=\frac{d(T_i)(n_i-d(T_i))}{n_i} \text{ \quad and \quad } \delta(S)=\sum_i\delta(T_i).$$

For $T_i\neq C_{i,0}$, let $\omega_{T_i}$ be the fundamental weight in the rational root lattice generated by $C_{i,1},\ldots,C_{i,n_i-1}$, normalized by $\omega_{T_i}\cdot T_i=-1$ and by zero intersection with every other nonidentity component; put $\omega_{C_{i,0}}=0$.  
The supporting class associated with $S$ is
\[
   H_S=C_0+\chi F-\sum_i\omega_{T_i}.
\]
By \cite[Propositions~4.1 and~4.3]{LLL26}, the nonvertical extremal rays of the dual of the fibration cone are generated by the classes $H_S$, and $H_S^2=\chi-\delta(S)$.  Moreover, $\delta(\pi)=\max_S\delta(S)$.  Since $\cC_\pi=\NE(X)$ at the boundary, these are precisely the nonvertical extremal rays of $\Nef(X)$.  
Assume now that $\delta(\pi)=\chi$.
Such a ray is isotropic exactly when the selected component has maximal cyclic distance in every reducible fiber.  An even cycle has one component of maximal distance, whereas an odd cycle has two.  Hence, if $o$ of the integers $n_i$ are odd, there are $2^o$ nonvertical isotropic extremal rays.
By the Hodge index theorem, every nonzero nef class of square zero spans
an extremal ray of the nef cone. Since the fiber class is semiample,
it remains to study the nonvertical isotropic rays described above.

\subsection{Vanishing of the Mordell--Weil group}

\begin{lemma}
\label{lem:MW-zero}
Assume that $\MW(\pi)$ is finite and  $\delta(\pi)<2\chi$.  
Then $\MW(\pi)=0$.  Consequently, the classes $C_0$, $F$, and $C_{i,j}$, for $1\leq i\leq s$ and $1\leq j\leq n_i-1$, form an integral basis of $\Pic(X)$.
\end{lemma}

\begin{proof}
Let $P$ be a nonzero torsion section.  By Shioda's height formula (cf. \cite[Lemma 8.2 and Theorem~8.6]{Shi90} or  \cite[Theorems 6.20 and 6.24]{SS19}), we have
\[
  0=\langle P,P\rangle
   =2\chi+2(P\cdot C_0)-\sum_i\operatorname{contr}_{F_i}(P).
\]
We refer to \cite[Definition 6.23]{SS19} for the definition of the local contribution $\operatorname{contr}_{F_i}(P)$.
At a fiber of type $I_{n_i}$, $\operatorname{contr}_{F_i}(P)$ is at most $\lfloor n_i^2/4\rfloor/n_i$.  Since two distinct sections have nonnegative intersection, the right-hand side is at least $2\chi-\delta(\pi)>0$, a contradiction. Thus $\delta(\pi)<2\chi$ excludes nontrivial torsion, and the finiteness assumption implies that $\MW(\pi)$ is trivial.

Let $T\subset\operatorname{NS}(X)$ be the trivial lattice generated by the zero section, a fiber, and the nonidentity components of the reducible fibers.  Shioda's exact sequence $\operatorname{NS}(X)/T\simeq\MW(\pi)$ \cite[Theorem~1.3]{Shi90} now gives $\operatorname{NS}(X)=T$.  The displayed generators are independent, and $q(X)=0$, so they form an integral basis of $\Pic(X)$.
\end{proof}

\begin{remark}
In particular, by Theorem \ref{NE-thm} (2),  every surface covered by the sufficient Mori dream criterion $N(\pi)\leq2\chi+3$ has trivial Mordell--Weil group.
\end{remark}

\subsection{Primitive integral generators}

\begin{proposition}\label{prop:primitive-generator}
Assume that $\delta(\pi)=\chi$, and let $S$ be an isotropic selection.  Put $q_i=2$ when $n_i$ is even and $q_i=n_i$ when $n_i$ is odd, and set $r=\lcm(q_1,\ldots,q_s)$.  Then the primitive integral generator of $\R_{\geq0}H_S$ is $D_S=rH_S$, and $D_S\cdot F=r$.
\end{proposition}

\begin{proof}
If the selected component in a fiber $I_n$ is $C_k$, the denominator of the corresponding fundamental weight in the root lattice $A_{n-1}$ is $n/\gcd(n,k)$.  
For a component of maximal cyclic distance, this denominator is two when $n$ is even and $n$ when $n$ is odd.  
By Lemma \ref{lem:MW-zero}, the zero section, the fiber class, and the nonidentity components form an integral basis of $\Pic(X)$. 
So the least positive integer for which a multiple of $H_S$ is integral is $r$.

It remains to check primitivity.  If $rH_S=dL$ for an integral class $L$ and an integer $d>1$, then $L=(r/d)H_S$.  
Since $H_S\cdot F=1$ by \cite[Proposition~4.3]{LLL26}, one has $L\cdot F=r/d\in\Z$.  
Hence $r/d$ is a smaller positive integer for which $(r/d)H_S$ is integral, a contradiction.  
Therefore $D_S=rH_S$ is primitive, and $D_S\cdot F=r$.
\end{proof}

\begin{corollary}
\label{cor:index-parity}
For every boundary isotropic ray, $r(\chi-2)$ is even.  In particular, if every reducible fiber has odd multiplicity, then $\chi$ is even.
\end{corollary}

\begin{proof}
Riemann--Roch implies that $D_S\cdot(D_S-K_X)$ is even.  Since $D_S^2=0$ and $K_X\cdot D_S=r(\chi-2)$, the first assertion follows.  If all $n_i$ are odd, then $r$ is odd, so $\chi-2$ is even.
\end{proof}

\subsection{An effective representative}

Fix a fiber $I_n=C_0^{\mathrm{fib}}+C_1+\cdots+C_{n-1}$ and select $C_k$, with $1\leq k\leq n-1$.  Here $C_0^{\mathrm{fib}}$ denotes the identity component of this fiber.  Define
\begin{equation}
\label{eq:Pnk}
 P_{n,k}=\frac{k(n-k)}{n}C_0^{\mathrm{fib}}
 +\sum_{j=1}^{k-1}\frac{(k-j)(n-k)}{n}C_j
 +\sum_{j=k+1}^{n-1}\frac{k(j-k)}{n}C_j.
\end{equation}

\begin{lemma}\label{lem:effective-cycle}
With the notation above, the class of $P_{n,k}$ is $\delta(C_k)F-\omega_{C_k}$.  Thus $P_{n,k}$ is effective, its coefficient at $C_k$ is zero, and all its other coefficients are positive.

If $S=\{T_1,\ldots,T_s\}$ is isotropic and $P_i$ denotes the corresponding divisor in $F_i$, then $C_0+\sum_iP_i$ represents the class $H_S$.  Consequently, $D_S=rH_S$ has an effective integral representative whose reduced support is a connected tree of rational curves, obtained by attaching the chain $F_i-T_i$ to $C_0$ for every $i$.
\end{lemma}

\begin{proof}
The $(j,k)$-entry of the inverse Cartan matrix of type $A_{n-1}$ is $$\frac{\min\{j,k\}(n-\max\{j,k\})}{n},$$
 this is the matrix used in \cite[Lemma~4.2]{LLL26}. 
  Subtracting the $k$-th fundamental weight from $\delta(C_k)F$ gives the coefficients in \eqref{eq:Pnk}.  
  Since an isotropic selection satisfies $\sum_i\delta(T_i)=\chi$, the class $C_0+\sum_iP_i$ is $H_S$.  
  Removing a component of maximal distance from an $I_n$-cycle leaves a chain containing the identity component; this chain meets the zero section. 
   The description of the reduced support is therefore immediate.
\end{proof}

\begin{remark}\label{rem:null-cycle}
Every irreducible component $C$ of $\Supp(D_S)$ satisfies $D_S\cdot C=0$.  In each reducible fiber, the selected component $T_i$ is the only component on which $D_S$ has positive degree, and $D_S\cdot T_i=r$.
\end{remark}

\subsection{A restriction on odd fibers}

\begin{lemma}
\label{lem:odd-reciprocals}
Assume that $\delta(\pi)=\chi$, and let $O$ be the set of indices for which $n_i$ is odd.  Then
\begin{equation}
\label{eq:reciprocal-identity}
   N(\pi)-4\chi=\sum_{i\in O}\frac1{n_i}.
\end{equation}
The right-hand side is a nonnegative integer.  If $o=|O|$, the following consequences hold:
\begin{enumerate}
\item[(i)] all $n_i$ are even if and only if $N(\pi)=4\chi$;
\item[(ii)] if an odd fiber occurs, then $N(\pi)\geq4\chi+1$ and $o\geq3$;
\item[(iii)] $N(\pi)-4\chi\equiv o\pmod2$;
\item[(iv)] if $o=3$, then the three odd fibers are all of type $I_3$;
\item[(v)] the case $o=4$ is impossible.
\end{enumerate}
\end{lemma}

\begin{proof}
For even $n$, one has $4\lfloor n^2/4\rfloor/n=n$, whereas for odd $n$ the same expression is $n-1/n$.  Summing and using $\delta(\pi)=\chi$ gives \eqref{eq:reciprocal-identity}. 
 If the sum is positive, it is at least one; because every summand is at most $1/3$, at least three odd fibers are required.

The parity assertion is immediate: $N(\pi)=\sum_i n_i$ has the same parity as the number $o$ of odd summands, while $4\chi$ is even.  If $o=3$, the positive integral sum is at most one and therefore equals one.  Equality in $1/n_1+1/n_2+1/n_3\leq1$ forces $n_1=n_2=n_3=3$.  If $o=4$, the sum is a positive integer at most $4/3$, hence it would be one, contrary to the parity assertion.
\end{proof}

\begin{corollary}\label{cor:at-most-three}
If $\delta(\pi)=\chi$ and $\pi$ has at most three reducible fibers, then every reducible fiber has even multiplicity.
\end{corollary}

\begin{proof}
If an odd fiber occurred,  Lemma \ref{lem:odd-reciprocals} would force exactly three odd fibers, all of type $I_3$.  With no additional reducible fiber, their total contribution to $\delta(\pi)$ would be two, contradicting $\chi\geq3$.
\end{proof}
\subsection{The obstruction to semiampleness}

Let $S$ be an isotropic selection and let $D=D_S$ be the primitive effective divisor constructed above.  Write $D_{\mathrm{red}}$ for its reduced support and set
\[
   U_D=\ker\bigl(\Pic(D)\longrightarrow\Pic(D_{\mathrm{red}})\bigr).
\]
Since $D_{\mathrm{red}}$ is a tree of rational curves, a line bundle of degree zero on every component of $D_{\mathrm{red}}$ restricts trivially to $D_{\mathrm{red}}$.  Let $\mathcal J$ be the nilradical of $\OO_D$.  In the exact sequence $1\to1+\mathcal J\to\OO_D^*\to\OO_{D_{\mathrm{red}}}^*\to1$, the map on global units is surjective because both schemes are connected and projective.  Since $\mathcal J$ is nilpotent and the ground field has characteristic zero, the logarithm identifies $1+\mathcal J$ with the additive sheaf $\mathcal J$.  It follows that $U_D$ is the vector group associated with the finite-dimensional vector space $H^1(D,\mathcal J)$.  In particular, it has no nontrivial torsion.

The following criterion shows that semiampleness is detected by the normal
bundle $\OO_D(D)$ on the single effective divisor $D$. Since $U_D$ has no
nontrivial torsion, passing to a positive multiple cannot remove this
obstruction.

\begin{theorem}\label{thm:normal-obstruction}
The following conditions are equivalent:
\begin{enumerate}
\item $\kappa(X,D)>0$;
\item $D$ is semiample;
\item $h^0(X,\OO_X(D))\geq2$;
\item $\OO_D(D)\simeq\OO_D$;
\item $|D|$ contains a base-point-free pencil.
\end{enumerate}
If these conditions hold, then $|D|$ defines a morphism $f_D\colon X\to\PP^1$ with connected fibers, $D$ is a fiber of $f_D$, and
\[
   g(D)=1+\frac{r(\chi-2)}2.
\]
Moreover, $h^0(X,\OO_X(D))=2$.
\end{theorem}

\begin{proof}
By Remark \ref{rem:null-cycle}, the restriction $\OO_D(D)$ has degree zero on every irreducible component of $D_{\mathrm{red}}$, so its class belongs to $U_D$.

Suppose first that $D$ is semiample.  Choose $m>0$ such that $|mD|$ is base-point-free, and let $\varphi$ be the associated morphism.  Every component of $D$ is orthogonal to $D$, so $\varphi(D)$ is supported at one point.  More precisely, the scheme-theoretic image of the connected scheme $D$ is a connected zero-dimensional scheme and is therefore the spectrum of an Artinian local ring.  Every line bundle on such a scheme is trivial.  Writing $\OO_X(mD)=\varphi^*A$ for the hyperplane bundle on the image of $\varphi$, the restriction $\OO_D(mD)$ is the pullback of $A$ from this Artinian local scheme and is therefore trivial.  Thus the class of $\OO_D(D)$ is torsion in $U_D$, and hence is trivial.

Conversely, assume that $\OO_D(D)$ is trivial.  From
\[
   0\longrightarrow\OO_X\longrightarrow\OO_X(D)
   \longrightarrow\OO_D(D)\longrightarrow0
\]
and $H^1(X,\OO_X)=0$, the constant section of $\OO_D$ lifts to a section $t$ of $\OO_X(D)$.  Together with the canonical section $s_D$, it defines a base-point-free pencil: outside $D$ the section $s_D$ does not vanish, while $t|_D=1$.  Thus (4) implies (5), and (5) implies (3).

If $h^0(X,\OO_X(D))\geq2$, then $\kappa(X,D)=1$, since $D^2=0$ and $D$ is not numerically trivial.  A nonzero effective nef divisor of square zero on a smooth projective surface with $q=0$ is either semiample or has Iitaka dimension zero by \cite[Lemma~5.1.2.2]{ADHL15}.  This gives (3)$\Rightarrow$(1)$\Rightarrow$(2) and completes the equivalence.

The pencil gives a morphism $f\colon X\to\PP^1$.  In its Stein factorization $X\xrightarrow{g}B\xrightarrow{h}\PP^1$, the curve $B$ has genus zero because $q(X)=0$.  If $e=\deg h$, then $\OO_X(D)=g^*\OO_B(e)$.  The class $D$ is primitive, so $e=1$.  Hence the fibers are connected and $D$ is a fiber class.  Adjunction gives $2g(D)-2=D\cdot(D+K_X)=r(\chi-2)$.  Finally, $\OO_X(D)=f_D^*\OO_{\PP^1}(1)$ and $f_{D*}\OO_X=\OO_{\PP^1}$, so $h^0(X,\OO_X(D))=2$.
\end{proof}

\begin{definition}\label{def:eta}
The class $\eta_S=[\OO_{D_S}(D_S)]\in U_{D_S}$ is the \emph{isotropic semiampleness obstruction}.  By Theorem \ref{thm:normal-obstruction}, the ray $\R_{\geq0}H_S$ is semiample exactly when $\eta_S=0$.
\end{definition}

\begin{remark}
The obstruction cannot disappear after taking a positive multiple.  Indeed, if $h^0(X,\OO_X(mD_S))\geq2$ for some $m>0$, then $\kappa(X,D_S)>0$, and Theorem \ref{thm:normal-obstruction} shows that $D_S$ itself defines a base-point-free pencil.
\end{remark}

\subsection{Riemann--Roch and direct images}

\begin{proposition}\label{prop:isotropic-pushforward}
Let $D=D_S$ and $r=D\cdot F$.  For every $m\geq1$, one has $R^1\pi_*\OO_X(mD)=0$.  Hence $E_m=\pi_*\OO_X(mD)$ is a vector bundle on $\PP^1$ of rank $mr$ and degree
\begin{equation}\label{eq:Em-degree}
   \deg E_m=\chi\left(1-\frac{mr}{2}\right).
\end{equation}
If $E_m\simeq\bigoplus_{j=1}^{mr}\OO_{\PP^1}(a_{m,j})$, then
\begin{equation}\label{eq:splitting-sums}
   \sum_{j=1}^{mr}a_{m,j}=\chi\left(1-\frac{mr}{2}\right),
   \qquad
   h^0(X,\OO_X(mD))=\sum_{j=1}^{mr}\max\{a_{m,j}+1,0\}.
\end{equation}
Moreover, $\chi(\OO_X(mD))=\chi-mr(\chi-2)/2$.  If $D$ is not semiample, then the splitting of $E_m$ contains exactly one summand $\OO_{\PP^1}$ and every other summand has degree at most $-1$.
\end{proposition}

\begin{proof}
On every fiber $G$ of $\pi$, the restriction $\OO_G(mD)$ has nonnegative degree on each irreducible component and positive total degree $mr$.  On a smooth elliptic fiber, this gives $H^1(G,\OO_G(mD))=0$.  If $G$ is reducible, Serre duality and $\omega_G\simeq\OO_G$ reduce the claim to the absence of sections of $\OO_G(-mD)$.  Such a section vanishes on the unique component of negative degree and then vanishes successively on every adjacent degree-zero component.  The same conclusion is immediate on an irreducible nodal fiber.  Cohomology and base change now give $R^1\pi_*\OO_X(mD)=0$ and $\operatorname{rk}E_m=mr$.

Since $(K_X-mD)\cdot F=-mr<0$, one also has $H^2(X,\OO_X(mD))=0$.  Riemann--Roch gives $\chi(\OO_X(mD))=\chi-mr(\chi-2)/2$.  On $\PP^1$, the identity $\chi(E_m)=\operatorname{rk}E_m+\deg E_m$ yields \eqref{eq:Em-degree}, and the splitting formulas follow.

If $D$ is not semiample,  Theorem \ref{thm:normal-obstruction} gives $\kappa(X,D)=0$.  Since every $mD$ is effective, it follows that $h^0(X,\OO_X(mD))=1$ for every $m\geq1$.  The asserted form of the splitting is then immediate from Birkhoff--Grothendieck.
\end{proof}

\begin{remark}
A positive summand $\OO_{\PP^1}(a)$ with $a>0$ is sufficient, but not necessary, for $h^0(X,\OO_X(mD))\geq2$: two summands of degree zero also suffice.  The precise criterion is \eqref{eq:splitting-sums}.
\end{remark}
\subsection{Bounded cohomology property}
\label{Sect-BCP}
Throughout this subsection, a curve means a reduced and irreducible curve.
A very open problem in the theory of algebraic surfaces is the following bounded negativity conjecture (BNC for short).
\begin{conjecture}
\cite[Conjecture 1.1]{Bauer et al. 13}
Let $X$ be a smooth projective surface over $\bC$.
Then there exists an integer $b=b(X)\ge0$ such that $C^2\ge-b$ for every curve $C\subseteq X$.
\end{conjecture}
Bauer et al. \cite[Conjecture 2.5.3]{Bauer et al. 12} introduced the following terminology.
\begin{definition}
\label{BCP-defn}
A smooth projective surface $X$ is said to satisfy the bounded cohomology property (BCP for short), if there exists a constant $c_X>0$ such that $h^1(\mathcal O_X(C))\le c_Xh^0 (\mathcal O_X(C))$ for every curve $C$ on $X$.
\end{definition}
It turns out by Ciliberto et al. \cite{Ciliberto et al. 2017} that the BCP implies the BNC.
Note that a smooth projective surface satisfies the BCP provided that  $-K_X$ is psuedoeffective (cf \cite[Proposition 1.4]{HL26}).
In general, it is  an open problem to classify smooth projective surfaces with the BCP as posted by Ciliberto et al.\cite[Question 6]{Ciliberto et al. 2017}.
\begin{problem}
Classify all smooth projective surfaces with the BCP.
\end{problem}
Hua and Li \cite{HL26, Li21, Li23, Li26a} characterized smooth projective surfaces  $X$  with the BCP where either (i) $\rho(X)=2$ and $\NE(X)$ is rational polyhedral or (ii) $X$ is a geometrically ruled surface.
The classification of the BCP for surfaces with higher Picard numbers remains a widely open problem.
In this setting, one is generally forced to assume that $\NE(X)$ is rational polyhedral (cf. \cite[Question 1.7]{HL26}).
Recently, Hua and Li proved in  \cite[Theorem 1.9]{HL26} that every Mori dream surface satisfies the BCP.
Their proof hinges on the condition that every nef extremal ray has positive Iitaka dimension, see \cite[Remark 1.7]{Li26a} and also \cite[Theorem 1.6(a)]{Li26b} as follows. 
\begin{theorem}
\label{thm-BCP}
Let $X$ be a smooth projective surface with $q(X)=0$ such that $\NE(X)$ is rational polyhedral.
Let $\Nef(X)=\sum_{i=1}^n\bR_{\ge0}[H_i]$  where each $H_i$ is an effective divisor on $X$.
Suppose each $H_i$ has $\kappa(X,H_i)\ge1$.
Then $X$ satisfies the BCP.
\end{theorem}
Let $X$ be as in the setting of Theorem \ref{NE-thm}.
It is straightforward to see that $X$ satisfies the BCP whenever $\delta(\pi)<\chi$ (cf. \cite[Corollary 1.11]{Li26b}).
When $\delta(\pi)=\chi$, Theorem \ref{thm:all-even} and Proposition \ref{prop:odd-example-bcp} provide new evidences for such surfaces $X$ satisfying the BCP and answer \cite[Question 3.2]{HL26}.
 These surfaces need not be Mori dream surfaces; there even exist some nonvertical isotropic extremal rays of $\Nef(X)$ spanned by an effective divisor $D$ with $\kappa(X,D)=0$.
\section{Surfaces with only even reducible fibers}
\label{Sect-alleven}
\subsection{Nef isotropic divisors on surfaces}
\begin{theorem}
\label{thm:all-even}
Assume that $\MW(\pi)$ is finite, $\delta(\pi)=\chi$ and that every reducible fiber is of type $I_{2m_i}$.  Let $D$ be the primitive generator of the unique nonvertical isotropic ray.  Then
\begin{equation}
\label{eq:D-even}
D=2C_0+\sum_{i=1}^s\Bigg(
 m_iC_{i,0}
 +\sum_{j=1}^{m_i-1}(m_i-j)C_{i,j}
 +\sum_{j=m_i+1}^{2m_i-1}(j-m_i)C_{i,j}
\Bigg),
\end{equation}
where the omitted component of $F_i$ is $C_{i,m_i}$.  The divisor $D$ is base-point-free and defines a second fibration $f_D\colon X\to\PP^1$ whose general fiber has genus $\chi-1$.
As a result, $X$ satisfies the BCP.
\end{theorem}
\begin{proof}
The component of maximal distance in an even cycle is unique.  By Proposition \ref{prop:primitive-generator}, the primitive index is $r=2$, and \eqref{eq:D-even} follows from \eqref{eq:Pnk}. 
 Since $D\cdot F=2$, the divisor $K_X-D$ has negative intersection with $F$, so $H^2(X,\OO_X(D))=0$.  By Riemann-Roch theorem, we have
\begin{equation*}
\begin{split}
\chi(\OO_X(D)) & =\chi-\frac{K_X\cdot D}{2}
\\&=2,
\end{split}
\end{equation*}
where we use $K_X\sim (\chi-2)F$.
Thus, $h^0(X,\OO_X(D))=2+h^1(X,\OO_X(D))\geq2$.  The result follows from Theorem \ref{thm:normal-obstruction}; the genus is $1+2(\chi-2)/2=\chi-1$.
Theorem \ref{thm-BCP} implies that $X$ satisfies the BCP.
\end{proof}

\begin{corollary}\label{cor:three-fibers}
\Cref{prob:isotropic} has an affirmative answer whenever $\pi$ has at most three reducible fibers.
\end{corollary}

\begin{proof}
This follows from Corollary \ref{cor:at-most-three} and Theorem \ref{thm:all-even}.
\end{proof}

The one-fiber boundary is particularly rigid.  The formula for $\delta(\pi)$ in \cite[Proof of Corollary~1.4]{LLL26} shows that the equality $\delta(\pi)=\chi$ for a unique reducible fiber $I_n$ forces $n=4\chi$.

\begin{corollary}\label{cor:I4chi}
Assume that $\delta(\pi)=\chi$ and that the unique reducible fiber is $F_1=C_{1,0}+\cdots+C_{1,4\chi-1}$.  The primitive isotropic generator is
\begin{equation}\label{eq:D-I4chi}
D=2C_0+2\chi C_{1,0}
 +\sum_{j=1}^{2\chi-1}(2\chi-j)C_{1,j}
 +\sum_{j=2\chi+1}^{4\chi-1}(j-2\chi)C_{1,j}.
\end{equation}
The component $C_{1,2\chi}$ is omitted, and $D$ is a fiber of a base-point-free pencil of genus $\chi-1$.
\end{corollary}

\begin{proof}
Apply Theorem  \ref{thm:all-even} with $m_1=2\chi$.
\end{proof}

\begin{remark}
The uniform Mori dream range $N(\pi)\leq2\chi+3$ lies strictly inside the region $\delta(\pi)<\chi$; see \cite[Proof of Lemma~4.5]{LLL26}. 
Theorem~\ref{thm:all-I2-mds} below proves that Mori dream surfaces do occur
on the boundary $\delta(\pi)=\chi$. For other configurations, the Mori dream
property also requires semiampleness of the big extremal nef rays.
\end{remark}
\subsection{Mori dream surfaces with reducible fibers of type \texorpdfstring{$I_2$}{I2}}
\label{Sect-all-I_2}
We now specialize to the case in which every reducible fiber has two components.  Write
$F_i=A_i+E_i$, where $A_i=C_{i,0}$ is the identity component and $E_i=C_{i,1}$ is the other component.  A selection is determined by a subset $J\subseteq\{1,\ldots,s\}$: for $i\in J$ we select $E_i$, whereas for $i\notin J$ we select $A_i$.  We denote the corresponding supporting class by $H_J$.

\begin{theorem}\label{thm:all-I2-mds}
Assume that $\MW(\pi)$ is finite, every reducible fiber of $\pi$ is of type $I_2$, and $s\leq2\chi$.  Then $\OO_X(2H_J)$ is globally generated for every subset $J\subseteq\{1,\ldots,s\}$.  Consequently, $X$ is a Mori dream surface.
In particular, if $\delta(\pi)=\chi$, then $s=2\chi$ and $X$ is a Mori dream surface.  The unique nonvertical isotropic extremal ray is generated by a semiample divisor defining a second fibration of genus $\chi-1$.
\end{theorem}

\begin{proof}
Since each reducible fiber is of type $I_2$, one has $\delta(\pi)=s/2$.  Thus the assumption $s\leq2\chi$ gives $\delta(\pi)\leq\chi$, and Theorem \ref{NE-thm} (1)  shows that $\NE(X)=\cC_\pi$.  The extremal rays of $\Nef(X)$ are therefore the fiber ray and the rays generated by the classes $H_J$.  The fiber class is semiample, so it remains to consider $H_J$.
Put $k=|J|$ and $d=2\chi-k$, which is nonnegative.  In a fiber of type $I_2$, the fundamental weight associated with $E_i$ is $E_i/2$, whereas the identity component contributes no correction.  Hence
\begin{equation}\label{eq:I2-HJ}
  H_J=C_0+\chi F-\frac12\sum_{i\in J}E_i,
  \qquad
  2H_J=M_J+dF,
\end{equation}
where $M_J=2C_0+\sum_{i\in J}A_i$.  We shall prove that $\OO_X(M_J+dF)$ is globally generated.
First consider the direct image $V_J=\pi_*\OO_X(M_J)$.  The restriction of $\OO_X(M_J)$ to every fiber has total degree two.  On the fiber $F_i=A_i+E_i$, its multidegree is $(0,2)$ when $i\in J$ and $(2,0)$ when $i\notin J$.  In particular, the degree is nonnegative on every irreducible component and positive on at least one component.  Serre duality on the fibers therefore gives $H^1(G,\OO_G(M_J))=0$ for every fiber $G$ of $\pi$.  Cohomology and base change imply that $R^1\pi_*\OO_X(M_J)=0$ and that $V_J$ is a vector bundle of rank two.
The intersections $M_J^2=-4\chi+2k=-2d$ and $K_X\cdot M_J=2(\chi-2)$ give
\[
  \chi(\OO_X(M_J))
  =\chi+\frac{M_J\cdot(M_J-K_X)}2
  =2-d.
\]
Since $\chi(V_J)=\operatorname{rk}(V_J)+\deg V_J$ on $\PP^1$, it follows that $\deg V_J=-d$.
The canonical section of the effective divisor $M_J$ induces a morphism $\OO_{\PP^1}\to V_J$.  Its restriction to every fiber is nonzero, because $M_J$ contains no entire fiber.  The morphism is therefore a line subbundle, and its quotient is a line bundle of degree $-d$.  We obtain an exact sequence
\[
  0\longrightarrow\OO_{\PP^1}\longrightarrow V_J
  \longrightarrow\OO_{\PP^1}(-d)\longrightarrow0.
\]
This sequence splits, since
$\operatorname{Ext}^1(\OO_{\PP^1}(-d),\OO_{\PP^1})
 =H^1(\PP^1,\OO_{\PP^1}(d))=0$.  Consequently,
\begin{equation}\label{eq:I2-splitting}
  V_J\simeq\OO_{\PP^1}\oplus\OO_{\PP^1}(-d),
  \qquad
  \pi_*\OO_X(2H_J)\simeq
  \OO_{\PP^1}(d)\oplus\OO_{\PP^1},
\end{equation}
where the second isomorphism follows from \eqref{eq:I2-HJ} and the projection formula.
It remains to verify relative generation.  The line bundle $\OO_G(M_J)$ is globally generated on every fiber $G$.  This is standard for a degree-two line bundle on a smooth elliptic fiber, and the irreducible nodal case follows from the normalization sequence.  On a fiber $I_2=A_i\cup E_i$ with multidegree $(0,2)$, the restriction to $A_i$ is trivial.  One section is nonzero on $A_i$, while a second section vanishes on $A_i$ and restricts to the section of $\OO_{E_i}(2)$ whose zeros are the two nodes.  These sections have no common zero; the case of multidegree $(2,0)$ is symmetric.  Thus the evaluation map
$\pi^*V_J\to\OO_X(M_J)$ is surjective.
The bundle $V_J(d)\simeq\OO_{\PP^1}(d)\oplus\OO_{\PP^1}$ is globally generated because $d\geq0$.  Twisting the relative evaluation map by $\pi^*\OO_{\PP^1}(d)$ now shows that $\OO_X(2H_J)$ is globally generated.
Every extremal ray of the nef cone is therefore semiample.  Since $q(X)=0$ and $\NE(X)$ is rational polyhedral, the Mori dream surface criterion in \cite[Propositions~2.1 and~2.2]{LLL26} applies and proves the result.  In the boundary case $s=2\chi$, the isotropic assertion and the genus of the second fibration follow from \cref{thm:all-even}.
\end{proof}

\begin{corollary}\label{cor:all-I2-MW}
Let $\pi\colon X\to\PP^1$ be a semistable Jacobian elliptic surface with $\chi(\OO_X)=\chi\geq3$, without assuming that $\MW(\pi)$ is finite.  Suppose that every reducible fiber is of type $I_2$ and that $s\leq2\chi$.  Then $X$ is a Mori dream surface if and only if $\MW(\pi)$ is finite.
\end{corollary}

\begin{proof}
The sufficiency is Theorem \ref{thm:all-I2-mds}.  Conversely, a Mori dream surface has rational polyhedral Mori cone, and infinitely many sections would give infinitely many negative curves spanning distinct extremal rays.  Thus $\MW(\pi)$ must be finite; see also \cite[Proposition~3.1]{LLL26}.
\end{proof}

\begin{remark}\label{rem:I2-blocks}
The direct-image argument goes beyond the range in which the rational form of Artin's criterion is immediately effective.  
Indeed, assume $k<2\chi$, so that $H_J$ is big.  Apart from isolated $(-2)$-curves, the main connected component of its null locus is the star formed by $C_0$ and the curves $A_i$ with $i\in J$.  Its fundamental cycle is $C_0+\sum_{i\in J}A_i$ when $k\leq\chi$, and $2C_0+\sum_{i\in J}A_i$ when $\chi<k<2\chi$.  The latter cycle has arithmetic genus $k-\chi-1$, which is positive as soon as $k\geq\chi+2$.  Thus the Artin argument does not by itself settle all these rays, whereas Theorem \ref{thm:all-I2-mds} proves their semiampleness uniformly.
\end{remark}
\begin{theorem}\label{thm:example-8I_2}
There exists a semistable Mori dream Jacobian elliptic surface 
$
\pi:X\to  \PP^1
$
with $
\chi(\mathcal O_X)=\delta(\pi)=4, \kappa(X)=1$, $\rho(X)=10$, and singular-fiber configuration $8I_2+32I_1$.
\end{theorem}
\begin{proof}
We first show that the stratum of semistable Jacobian elliptic surfaces
with $\chi=4$ and singular-fiber configuration $8I_2+32I_1$ is nonempty.
Consider the Weierstrass family
\[
y^2=x^3-3P^2x+\bigl(2P^3+S^2Q_0\bigr),
\qquad P,S,Q_0\in H^0(\PP^1,\OO_{\PP^1}(8)).
\]
Its discriminant and $c_4$ are
\[
\Delta=-16\cdot27\,S^2Q_0\bigl(4P^3+S^2Q_0\bigr),
\qquad c_4=144P^2.
\]
On the affine chart with coordinate $t$, take
$P=1$, $S=t^8-1$, and $Q_0=t^8+1$, and homogenize to degree eight.
The polynomials $S$ and $Q_0$ have simple, disjoint zero sets. Moreover,
\[
4+S^2Q_0=t^{24}-t^{16}-t^8+5
\]
has twenty-four simple zeros, disjoint from those of $S$ and $Q_0$.
Indeed, the polynomial $z^3-z^2-z+5$ has no common root with
its derivative $(3z+1)(z-1)$ and does not vanish at $z=0,1,-1$.
At every finite zero of $\Delta$ one has $c_4\neq0$, and the fiber at
infinity is smooth. Thus the minimal resolution has exactly eight
$I_2$-fibers and thirty-two $I_1$-fibers, and $\chi=4$.

Let $\mathcal M_4$ be the moduli space of Jacobian elliptic surfaces
with $\chi=4$, and let $U$ be its locus with nonconstant $j$-invariant.
We now work with the full stratum with the above singular-fiber
configuration, rather than only with the displayed Weierstrass family.
Label the eight points supporting the $I_2$-fibers. Since these points
are allowed to move, the conditions that the discriminant vanish to
order at least two impose at most eight conditions on the
$38$-dimensional space $\mathcal M_4$.
Restricting to the nonempty open locus with exactly the stated
semistable configuration, every component through our example
therefore has dimension at least $30$.

The zero section, a fiber, and the eight nonidentity components give
ten independent divisor classes. On the other hand, Kloosterman's
dimension theorem gives
\[
\dim(\operatorname{NL}_{11}\cap U)=29;
\]
see \cite[Theorem~1.1]{Klo07}. Hence such a component cannot be
contained in $\operatorname{NL}_{11}$, and a very general member $X$
has $\rho(X)=10$.
The Shioda--Tate formula gives Mordell--Weil rank zero, and
Lemma~\ref{lem:MW-zero} excludes nontrivial torsion because
$\delta(\pi)=8/2=4<2\chi$. Thus $\MW(\pi)=0$.
Finally, $K_X\sim2F$, so $\kappa(X)=1$, and
Theorem~\ref{thm:all-I2-mds} shows that $X$ is a Mori dream surface.
\end{proof}

\subsection{Big nef divisors on surfaces with even reducible fibers}
\label{sec:even-big}
Theorem~\ref{thm:all-even} concerns the isotropic extremal ray at
$\delta(\pi)=\chi$.  We now study the remaining nef rays.  Throughout
this section, the reducible fibers have types $I_{2m_1},\ldots,I_{2m_s}$,
the Mordell--Weil group is finite, and
$T:=\sum_i m_i\leq2\chi$, where $\chi\geq3$.
Thus $N(\pi)=2T$, $\delta(\pi)=T/2$, and
Theorem~\ref{NE-thm} gives $\NE(X)=\cC_\pi$.
Write $d_i=d(T_i)$ for the cyclic distance of the component chosen in
the $i$-th fiber.  Then $0\leq d_i\leq m_i$ and
\[
 H_S^2=\chi-\sum_i\frac{d_i(2m_i-d_i)}{2m_i}.
\]
We call a choice symmetric if $d_i\in\{0,m_i\}$ for every $i$.

\begin{theorem}\label{evenbig:main}
Under the assumptions above, the following statements hold.
\begin{enumerate}
\item If $N(\pi)\leq2\chi+4$, then $X$ is a Mori dream surface.
\item If some $m_i\geq2$ satisfies $T-m_i>\chi$, then $X$ has a
big and nef divisor which is not semiample.  In particular, $X$ is
not a Mori dream surface.
\item If at least one reducible fiber is of type $I_4$, then $X$ is
a Mori dream surface if and only if $N(\pi)\leq2\chi+4$.
\end{enumerate}
\end{theorem}

The first statement enlarges the sufficient range in Theorem~\ref{NE-thm} for even fibers.  The second distinguishes
semiampleness of the isotropic ray from semiampleness of the big rays:
it can fail for the latter even when every reducible fiber is even.

We begin with the symmetric choices.  The following lemma extends the
rank-two argument of Theorem~\ref{thm:all-I2-mds} to arbitrary even
fibers and does not require $H_S^2=0$.

\begin{lemma}\label{evenbig:symmetric}
If $S$ is symmetric, then $2H_S$ is an integral divisor and
$\OO_X(2H_S)$ is globally generated.  More precisely, let
$J=\{i:d_i=m_i\}$ and put $e=2\chi-\sum_{i\in J}m_i$.  Then
\[
 \pi_*\OO_X(2H_S)\simeq
 \OO_{\PP^1}(e)\oplus\OO_{\PP^1}.
\]
\end{lemma}

\begin{proof}
For $i\in J$, let $P_i=P_{2m_i,m_i}$ be the rational divisor in
\eqref{eq:Pnk}, and set $M=2C_0+2\sum_{i\in J}P_i$.
The divisor $M$ is effective and integral, and
$2H_S\sim M+eF$.  It contains no entire fiber.
Its restriction to each fiber has total degree two, concentrated on
the component selected by $S$.  In particular,
$R^1\pi_*\OO_X(M)=0$, by fiberwise duality, and
$V=\pi_*\OO_X(M)$ is a vector bundle of rank two.
The identities $M^2=-2e$ and $K_X\cdot M=2(\chi-2)$ give
$\chi(\OO_X(M))=2-e$, so $\deg V=-e$.

The canonical section of $M$ is nonzero on every fiber, because $M$
contains no entire fiber.  It therefore defines a line subbundle
$\OO_{\PP^1}\subset V$ with quotient $\OO_{\PP^1}(-e)$.
Since $e\geq0$, the extension splits.  Twisting by
$\OO_{\PP^1}(e)$ proves the asserted formula.

It remains to check relative generation.  On a reducible fiber, the
complement of the selected component is a chain of rational curves
on which the line bundle has degree zero.  A section nonzero on this
chain extends across the selected component: its two endpoint values
can be prescribed in $\OO_{\PP^1}(2)$.  A second section vanishes on
the chain and has its two zeros at the attaching points.  These two
sections have no common zero.  The smooth and irreducible nodal
fibers are handled by the usual degree-two genus-one argument.
Thus $\pi^*V\to\OO_X(M)$ is surjective.  As
$V(e)\simeq\OO(e)\oplus\OO$ is globally generated, so is
$\OO_X(M+eF)$.
\end{proof}

The first estimate below specializes \cite[Proposition~4.4]{LLL26}.
The sharper estimate for a cycle with coefficient two at $C_0$
will allow us to enlarge the sufficient range.

\begin{lemma}\label{evenbig:genus}
Let $H_S$ be big.  If $Z$ is an effective integral cycle supported on
$\Null(H_S)$ and its coefficient at $C_0$ is $z\geq1$, then
\[
 Z^2+K_X\cdot Z
 \leq-\chi z^2+(\chi-2)z
       +2T\left\lfloor\frac{z^2}{4}\right\rfloor.
\]
For $z=2$, the sharper bound $p_a(Z)\leq\sum_i d_i-\chi-1$ holds.
\end{lemma}

\begin{proof}
Since $\NE(X)=\cC_\pi$, the null locus consists of $C_0$ and the
unselected fiber components.  Indeed, a curve orthogonal to a big
nef class has negative self-intersection and spans an extremal ray
of the Mori cone.  Moreover, $H_S\cdot C_0=0$, and in each fiber
$H_S$ has intersection one with the selected component and zero
with every other component.

Consider a fiber with $d_i>0$.  Removing the selected component
leaves two paths from that component to the identity component,
of lengths $d_i$ and $2m_i-d_i$.  Assign coefficient zero to the
removed component and let $u$ be the coefficient of the identity
component.  The integral coefficient increments along the two
paths are $a_1,\ldots,a_{d_i}$ and
$b_1,\ldots,b_{2m_i-d_i}$, respectively.  Both sums equal $u$.
The contribution of this fiber, including its intersection with
$zC_0$, is
\[
 G_i=2zu-\sum_j a_j^2-\sum_j b_j^2
     =\sum_j(za_j-a_j^2)+\sum_j(zb_j-b_j^2).
\]
For every integer $v$, one has
$zv-v^2\leq\lfloor z^2/4\rfloor$.  Hence
$G_i\leq2m_i\lfloor z^2/4\rfloor$.
When $z=2$, we instead write
\[
 G_i=\sum_j(3a_j-a_j^2)+\sum_j(b_j-b_j^2)
 \leq2d_i,
\]
using $3v-v^2\leq2$ and $v-v^2\leq0$ for integral $v$.
If $d_i=0$, the remaining chain is disjoint from $C_0$ and its
contribution is nonpositive.  Finally,
$C_0^2=-\chi$ and $K_X\cdot Z=(\chi-2)z$.
Summing proves both assertions.
\end{proof}

\begin{proposition}\label{evenbig:mds-bound}
If $T\leq\chi+2$, then $X$ is a Mori dream surface.
\end{proposition}

\begin{proof}
The fiber class is semiample.  Symmetric selections are covered by
Lemma~\ref{evenbig:symmetric}.  Consider any other selection $S$.
Then $\sum_i d_i\leq T-1$, and
$H_S^2\geq\chi-T/2\geq(\chi-2)/2>0$.
Its null locus has negative definite intersection matrix.

We show that every nonzero effective integral cycle $Z$ supported
on this null locus has arithmetic genus at most zero.  If its
coefficient $z$ at $C_0$ is zero, this follows because its support
is a disjoint union of $A$-chains.  If $z=1$, the first inequality
of Lemma~\ref{evenbig:genus} gives $p_a(Z)\leq0$.
If $z=2$, its second inequality gives
$p_a(Z)\leq T-\chi-2\leq0$.

Put $a=2\chi-T\geq\chi-2$.  For $z=2q\geq4$, the first
inequality of that lemma gives
\[
 p_a(Z)\leq1+(\chi-2)q-aq^2
 \leq1-(\chi-2)q(q-1)<0.
\]
For $z=2q+1\geq3$, it gives
\[
 p_a(Z)\leq q\bigl((\chi-2)-a(q+1)\bigr)
 \leq-(\chi-2)q^2<0.
\]
Artin's criterion therefore shows that each connected component of
the null locus is a rational exceptional configuration.  In particular,
$H^1(Z,\OO_Z)=0$ for every positive cycle supported there.
Choose a block for an integral multiple of $H_S$.
Each connected component of this block has vanishing first cohomology,
so the block criterion \cite[Proposition~2.6]{LLL26} gives an empty
stable base locus.  Thus $H_S$ is semiample; see also
\cite[Section~5.1.2]{ADHL15}.
Every extremal nef ray is consequently semiample.  Together with
$q(X)=0$ and rational polyhedrality, this proves the assertion.
\end{proof}

The following obstruction concerns all multiples of a supporting divisor.
Its essential feature is that the coefficient immediately below the
leading Weierstrass term always belongs to $H^0(\PP^1,\OO_{\PP^1}(\chi))$,
independently of the multiple.

\begin{lemma}\label{evenbig:coefficient-obstruction}
Let $S$ be a selection, and let $k\geq4$ be a common multiple of
$2m_1,\ldots,2m_s$.  Suppose that a section of $\OO_X(kH_S)$ restricts
nontrivially to $C_0$.  There is then a section
$a\in H^0(\PP^1,\OO_{\PP^1}(\chi))$ with the following properties.
\begin{enumerate}
\item If $S$ selects the opposite component $C_{j,m_j}$ of the fiber
above $p_j$, then $\operatorname{ord}_{p_j}(a)\geq m_j$.
\item If $S$ selects $C_{i,m_i-1}$, where $m_i\geq2$, then $a(p_i)\neq0$.
\end{enumerate}
\end{lemma}

\begin{proof}
Put $N=k/2$.  Since all correction divisors are effective, the given
section also belongs to
$H^0(X,\OO_X(kC_0+k\chi F))$.
Use a short Weierstrass equation $y^2=x^3+Ax+B$, with fundamental line
bundle $\OO_{\PP^1}(\chi)$.
The monomials in $x$ and $y$ of pole order at most $k$ at the origin
form the usual basis of the direct image of $\OO_X(kC_0)$.
The coefficient of $x^N$ is a constant and agrees, up to a fixed
nonzero scalar, with restriction to $C_0$.
After normalizing this constant to one, we can write the section as
\[
f=x^N+a(t)yx^{N-2}+
\text{terms of pole order at most }k-2.
\]
Because $yx^{N-2}$ has pole order $k-1$, its coefficient is a global
section of $\OO_{\PP^1}(\chi)$.

We first work at a fiber of type $I_{2m}$, with local parameter $t$.
After completing the base, choose a critical point $r(t)$ of the cubic
which specializes to its double root.  Since that root is nonzero,
this choice is possible and $r(0)\neq0$.  With $\xi=x-r(t)$, the equation
becomes
\begin{equation}\label{evenbig:local-Weierstrass}
y^2=\xi^3+3r(t)\xi^2+t^{2m}u(t),
\qquad r(0)u(0)\neq0.
\end{equation}
On the middle component of the minimal resolution,
$\operatorname{ord}(t)=1$ and
$\operatorname{ord}(\xi)=\operatorname{ord}(y)=m$.
The residues $U=\xi/t^m$ and $V=y/t^m$ satisfy
\[
V^2=3r(0)U^2+u(0).
\]
In particular, $1$ and $V$ are linearly independent over $\C(U)$:
the polynomial on the right has two distinct roots and is not a
square in $\C(U)$.

Write $f=P(\xi)+yQ(\xi)$ locally.  All coefficients are regular in $t$,
and the leading coefficient of $Q$, at $\xi^{N-2}$, remains $a(t)$.
For this middle-component valuation, no cancellation occurs between
the lowest-order terms: terms in the even and odd parts have residues
in $\C[U]$ and $V\C[U]$, respectively, and distinct powers of $U$ within
either part are linearly independent.
If the selected component is the opposite one, the coefficient of
$\omega_{C_m}$ on $C_m$ is $m/2$.  Thus
$\operatorname{ord}_{C_m}(f)\geq km/2=Nm$, and applying this to
$a(t)y\xi^{N-2}$ gives
\[
\operatorname{ord}_{t}(a)+(N-1)m\geq Nm.
\]
This proves the first assertion.

For the second assertion, take $m=m_i$ and set $r_0=r(0)$.
The normalization of the special Weierstrass cubic is given by
\[
x=z^2-2r_0,\qquad y=z(z^2-3r_0).
\]
Its two points above the node are $z=\lambda$ and $z=-\lambda$, where
$\lambda^2=3r_0$ and $\lambda\neq0$.
The identity component of the resolved fiber is this normalization.
For the selection $C_{m-1}$, the coefficients of the correction divisor
on the two components adjacent to the identity component are
$(m+1)/(2m)$ and $(m-1)/(2m)$.
Consequently, the restriction $f_0$ of $f$ to the identity component
vanishes at the two points above the node to orders at least
\[
\alpha=\frac{k(m+1)}{2m},\qquad
\beta=\frac{k(m-1)}{2m}.
\]
Indeed, at either intersection the resolved surface is smooth, and
vanishing along the neighboring component imposes the stated zero
on the identity component.
The integers $\alpha$ and $\beta$ sum to $k$, while $f_0$ is a monic
polynomial of degree $k$ in $z$.  Hence, after possibly exchanging
the two signs,
\[
f_0=(z-\lambda)^\alpha(z+\lambda)^\beta.
\]
Its coefficient of $z^{k-1}$ is therefore
$-(\alpha-\beta)\lambda=-k\lambda/m$.
On the other hand, the Weierstrass expression for $f$ identifies this
coefficient with $a(p_i)$: the term $x^N$ is even in $z$, the term
$yx^{N-2}$ is monic of degree $k-1$, and all remaining terms have degree
at most $k-2$.
Thus $a(p_i)=\pm k\lambda/m\neq0$.
\end{proof}

\begin{proposition}\label{evenbig:nonsemiample}
Assume that there is an index $i$ such that $m_i\geq2$ and
$T-m_i>\chi$.  Select $C_{i,m_i-1}$ in $F_i$ and the opposite component
$C_{j,m_j}$ in every $F_j$ with $j\neq i$, and denote the resulting
supporting class by $H_S$.  Then $H_S$ is nef and big but not semiample.
More precisely, if $M$ is a common multiple of $2m_1,\ldots,2m_s$, then
$C_0$ is contained in the base locus of $|\ell MH_S|$ for every
$\ell\geq1$.  In particular, $X$ is not a Mori dream surface.
\end{proposition}

\begin{proof}
The assumption $T\leq2\chi$ gives $\delta(\pi)=T/2\leq\chi$, so
Theorem~\ref{NE-thm} identifies the Mori cone with the fibration cone.
Thus $H_S$ is nef.  The correction formula gives
\[
H_S^2=\chi-\frac{T}{2}+\frac{1}{2m_i}>0,
\]
so it is big.
If a section of $\OO_X(\ell MH_S)$ did not vanish identically on $C_0$,
Lemma~\ref{evenbig:coefficient-obstruction} would provide a section
$a$ of $\OO_{\PP^1}(\chi)$ satisfying
\[
\operatorname{ord}_{p_j}(a)\geq m_j\quad(j\neq i),
\qquad a(p_i)\neq0.
\]
This is impossible, since $\sum_{j\neq i}m_j=T-m_i>\chi$.
The base locus assertion follows for every $\ell$.
If any integral multiple of $H_S$ were globally generated, so would
be a further multiple divisible by $M$, contrary to this assertion.
Hence $H_S$ is not semiample.
\end{proof}

\begin{corollary}\label{evenbig:I4-classification}
Under the standing assumptions, suppose that at least one reducible
fiber is of type $I_4$.  Then $X$ is a Mori dream surface if and only if
$T\leq\chi+2$.
In particular, if all the $s$ reducible fibers are of type $I_4$, then
$X$ is a Mori dream surface if and only if $2s\leq\chi+2$.
\end{corollary}

\begin{proof}
The positive assertion is Proposition~\ref{evenbig:mds-bound}.
If $T>\chi+2$, choose an index $i$ with \revision{$m_i=2$} and apply
Proposition~\ref{evenbig:nonsemiample}.
\end{proof}

\begin{proof}[Proof of Theorem~\ref{evenbig:main}]
Since $N(\pi)=2T$, the three statements follow, respectively, from
Propositions~\ref{evenbig:mds-bound} and~\ref{evenbig:nonsemiample}
and Corollary~\ref{evenbig:I4-classification}.
\end{proof}

\begin{corollary}\label{evenbig:boundary-I4}
Assume that $\delta(\pi)=\chi$ and that at least one reducible fiber
is of type $I_4$.  Then $X$ is not a Mori dream surface, although
every nef isotropic divisor is semiample and $X$ satisfies the
bounded cohomology property.
\end{corollary}

\begin{proof}
Here $T=2\chi>\chi+2$, because $\chi\geq3$.
Corollary~\ref{evenbig:I4-classification} excludes the Mori dream
property. 
Theorem~\ref{thm:all-even} proves semiampleness of the
unique nonvertical isotropic ray, and the fiber ray is semiample.
Every remaining extremal nef ray is big and hence has Iitaka dimension two. 
Theorem \ref{thm-BCP} implies that  $X$ satisfies the bounded cohomology property.
\end{proof}
\subsection{Proof Theorem \ref{Mainthm-even}}
\begin{proof}[Proof of Theorem \ref{Mainthm-even}]
It follows from Theorem  \ref{thm:all-even}, \ref{thm:all-I2-mds}, and \ref{evenbig:main}.
 \end{proof}
\section{A counterexample with odd reducible fibers}
\label{sec:odd-counterexample}

The first numerically possible boundary configuration containing odd reducible
fibers already behaves differently from the even-fiber case.  In fact, the
semiampleness statement fails for all the isotropic rays of a suitable surface
with configuration $I_4+3I_3$.

\begin{theorem}
\label{thm:odd-counterexample}
There exists a semistable Jacobian elliptic surface
$\pi\colon X\to\PP^1$ with $\chi(\OO_X)=3$, trivial Mordell--Weil group, and
singular-fiber configuration
\[
I_4+3I_3+23I_1
\]
such that none of the eight nonvertical isotropic extremal rays of
$\Nef(X)$ is semiample.

More precisely, if $D$ is the primitive integral generator of any of these
rays, then $h^0(X,\OO_X(D))=1$ and $\kappa(X,D)=0$.  Moreover,
\[
\pi_*\OO_X(D)\simeq
\OO_{\PP^1}\oplus\OO_{\PP^1}(-2)
\oplus\OO_{\PP^1}(-1)^{\oplus4}.
\]
\end{theorem}

\begin{proof}
We first construct one point of the relevant equisingular stratum at which
the space of sections can be computed exactly.  On the affine chart of the
base with coordinate $t$, consider
\begin{equation}
\label{eq:odd-Weierstrass}
y^2=x^3-3x+2+P(t)R(t),
\qquad
P(t)=t^4(t-1)^3(t-2)^3(t-3)^3,
\qquad
R(t)=t^5+t+1.
\end{equation}
After homogenization, the coefficients have degrees $12$ and $18$, so this
defines a Weierstrass fibration with $\chi=3$.  Up to a nonzero constant,
its discriminant is
\[
P(t)R(t)\bigl(4+P(t)R(t)\bigr).
\]
A direct exact computation shows that $R$ and $4+PR$ are squarefree and
that $P$, $R$, and $4+PR$ are pairwise coprime.  At each finite zero of the
discriminant the coefficient $c_4$ is nonzero, so all these singular fibers
are multiplicative.  The fiber over infinity is smooth.  The minimal
resolution therefore has one fiber of type $I_4$, three fibers of type
$I_3$, and twenty-three fibers of type $I_1$.  In particular,
$\delta(\pi)=1+3(2/3)=3=\chi$.

Write the components of the $I_4$-fiber over $t=0$ as
$\Theta_{0,0},\ldots,\Theta_{0,3}$, with $\Theta_{0,0}$ the identity
component.  For $a=1,2,3$, write the components of the $I_3$-fiber over
$t=a$ as $\Theta_{a,0},\Theta_{a,1},\Theta_{a,2}$.  The unique component
of maximal cyclic distance in the $I_4$-fiber is $\Theta_{0,2}$, whereas
each $I_3$-fiber has two components of maximal distance.  Thus the
isotropic selections are indexed by
$\boldsymbol\epsilon=(\epsilon_1,\epsilon_2,\epsilon_3)\in\{1,2\}^3$.

By Lemma \ref{lem:effective-cycle}, the corresponding primitive effective
cycle is
\begin{equation}
\label{eq:odd-D-epsilon}
D_{\boldsymbol\epsilon}
=
6C_0+6\Theta_{0,0}+3\Theta_{0,1}+3\Theta_{0,3}
+\sum_{a=1}^3
\bigl(4\Theta_{a,0}+2\Theta_{a,3-\epsilon_a}\bigr).
\end{equation}
The selected components $\Theta_{0,2}$ and
$\Theta_{a,\epsilon_a}$ do not occur.  Notice that
$D_{\boldsymbol\epsilon}\cdot F=6$, in agreement with
Proposition \ref{prop:primitive-generator}.

We next compute $H^0(X,\OO_X(D_{\boldsymbol\epsilon}))$.  We first recall
the splitting of $\pi_*\OO_X(6C_0)$.  Since $C_0\simeq\PP^1$ and
$C_0^2=-\chi=-3$, one has
$\OO_{C_0}(mC_0)\simeq\OO_{\PP^1}(-3m)$.  For every $m\geq2$, the
restriction sequence
\[
0\longrightarrow\OO_X((m-1)C_0)
\longrightarrow\OO_X(mC_0)
\longrightarrow\OO_{C_0}(mC_0)
\longrightarrow0
\]
may be pushed forward to the base.  Indeed,
$R^1\pi_*\OO_X((m-1)C_0)=0$: on every fiber the restriction has
positive total degree $m-1$, concentrated on the component meeting
$C_0$, and hence has no first cohomology.  Thus, writing
$E_m=\pi_*\OO_X(mC_0)$, we obtain
\[
0\longrightarrow E_{m-1}\longrightarrow E_m
\longrightarrow\OO_{\PP^1}(-3m)\longrightarrow0.
\]
The initial bundle is $E_1\simeq\OO_{\PP^1}$, since on every fiber
there is a unique section of $\OO_F(C_0\cap F)$.  The displayed
extension splits inductively.  Indeed, assuming
\[
E_{m-1}\simeq
\OO_{\PP^1}\oplus
\bigoplus_{k=2}^{m-1}\OO_{\PP^1}(-3k),
\]
its extension class belongs to
\[
\operatorname{Ext}^1
\bigl(\OO_{\PP^1}(-3m),E_{m-1}\bigr)
=
H^1\bigl(\PP^1,E_{m-1}(3m)\bigr),
\]
and this group vanishes because every summand of
$E_{m-1}(3m)$ has nonnegative degree.  Consequently,
\[
\pi_*\OO_X(mC_0)\simeq
\OO_{\PP^1}\oplus
\bigoplus_{k=2}^{m}\OO_{\PP^1}(-3k).
\]
For $m=6$ this gives
\[
\pi_*\OO_X(6C_0)\simeq
\OO_{\PP^1}\oplus\OO_{\PP^1}(-6)\oplus
\OO_{\PP^1}(-9)\oplus\OO_{\PP^1}(-12)\oplus
\OO_{\PP^1}(-15)\oplus\OO_{\PP^1}(-18).
\]
This splitting also has the familiar Weierstrass interpretation.  On
the generic elliptic fiber $F$, with origin $O=C_0\cap F$, the functions
$x$ and $y$ have poles of orders $2$ and $3$ at $O$.  Hence
\[
1,\quad x,\quad y,\quad x^2,\quad xy,\quad x^3
\]
form a basis of $H^0(F,\OO_F(6O))$ and correspond, respectively, to the
six summands of degrees
$0,-6,-9,-12,-15,-18$ above.  Thus a rational section of
$\OO_X(6C_0)$ over the generic point of the base may be written uniquely
as
\[
a_0(t)+a_1(t)x+a_2(t)y+a_3(t)x^2+a_4(t)xy+a_5(t)x^3.
\]
We now enlarge $6C_0$ to the divisor $D_{\boldsymbol\epsilon}$.
Its vertical part is supported on the four reducible fibers over
$t=0,1,2,3$.  The coefficient of the identity component is $6$ over
$t=0$ and $4$ over each of $t=1,2,3$.  Therefore, before imposing
the conditions coming from the remaining components, the coefficients
$a_i(t)$ are allowed to have poles of these orders at the corresponding
points of the base.  Set
\[
Q(t)=t^6(t-1)^4(t-2)^4(t-3)^4.
\]
Every candidate section can then be written in the form
\[
f=\sum_{i=0}^5\frac{N_i(t)}{Q(t)}\,b_i,
\qquad
(b_0,\ldots,b_5)=(1,x,y,x^2,xy,x^3).
\]
The degrees of the numerators are bounded, respectively, by
$18,12,9,6,3,0$.  Indeed, these bounds are obtained by adding
$\deg Q=18$ to the degrees $0,-6,-9,-12,-15,-18$ of the six summands
of $\pi_*\OO_X(6C_0)$.  Hence the space of candidate sections has
dimension
\[
19+13+10+7+4+1=54.
\]
It remains to impose regularity along the nonidentity components.
Near a fiber of type $I_n$ over $t=a$, put $s=t-a$ and translate $x$
by the double root of the nodal cubic.  After a formal change of
coordinates, the singularity takes the form
$uv=s^nh(s)$, with $h(0)\neq0$.  If
$E_1,\ldots,E_{n-1}$ are the exceptional components, their divisorial
valuations satisfy
\[
\nu_j(u)=j,\qquad \nu_j(v)=n-j,\qquad \nu_j(s)=1.
\]
For the $I_4$-fiber, the regularity conditions on the numerator give
the three valuation thresholds $3,6,3$.  At an $I_3$-fiber they give
$4,2$, or the reversed pair $2,4$, according to the selected maximal
component.

For completeness, the local expansions used in the calculation can be
specified explicitly. Put $B=PR$, $z=x-1$, and $Y=y/\sqrt{3}$, and set
\[
  w=z\sqrt{1+z/3},\qquad u=Y+w,\qquad v=Y-w.
\]
Then $uv=B/3$, $w=(u-v)/2$, and $Y=(u+v)/2$. The constant change
from $y$ to $Y$ does not change the dimension of the space of sections
and allows the computation to be carried out over $\mathbb Q$.
The inverse formal series is
\[
  z=\sum_{k\geq1}
  \frac{1}{k\,3^{k-1}}
  \binom{-k/2}{k-1}w^k.
\]
At $t=a$, write $s=t-a$ and $B(a+s)/3=s^nH(s)$.
For the component $E_j$, substitute
\[
  u=s^jU,\qquad v=s^{n-j}H(s)U^{-1}.
\]
The regularity conditions are obtained by setting to zero every
coefficient of $s^rU^\ell$ with $r$ below the corresponding valuation
threshold. Since these thresholds are at most six, only terms
through $w^5$ are needed.

Expanding the six basis elements in the local coordinates and imposing
these valuation inequalities gives, for every
$\boldsymbol\epsilon$, a rational matrix
$M_{\boldsymbol\epsilon}$ with $75$ rows and $54$ columns.  The exact
calculation gives
\begin{equation}
\label{eq:odd-rank-computation}
\operatorname{rank}
\bigl(M_{\boldsymbol\epsilon}\bmod 1000003\bigr)=53,
\qquad
M_{\boldsymbol\epsilon}c_{\mathrm{can}}=0,
\end{equation}
where $c_{\mathrm{can}}$ is the coefficient vector of the canonical
section $1\in H^0(X,\OO_X(D_{\boldsymbol\epsilon}))$.

The first equality shows that the rank over $\mathbb Q$ is at least
$53$, while the exact kernel vector shows that it is at most $53$.
Hence its rank over $\mathbb Q$ is exactly $53$.  Therefore
\[
h^0(X,\OO_X(D_{\boldsymbol\epsilon}))=1
\]
for all eight choices of $\boldsymbol\epsilon$.  The computation is
performed with exact rational arithmetic; the reduction modulo
$1000003$ is used only to certify a nonzero $53\times53$ minor.  The
calculation is contained in the ancillary file \href{https://github.com/alaface/jacobian-semiampleness}{check\_I4\_3I3\_semampleness.py}.

We now pass from this explicit point to a surface with trivial
Mordell--Weil group.  Let $\mathcal M_3$ denote the moduli space of
Jacobian elliptic surfaces with $\chi=3$, and let
$U\subset\mathcal M_3$ be the locus with nonconstant $j$-invariant.
Consider the incidence stratum with four distinct labeled points at
which the discriminant has orders at least $4,3,3,3$.  It is nonempty
by \eqref{eq:odd-Weierstrass}. 
We restrict this incidence stratum to the open locus of semistable
surfaces with exactly the singular-fiber configuration
$I_4+3I_3+23I_1$.
Since the four points are allowed to
move, these multiplicity conditions impose at most
\[
(4-1)+3(3-1)=9
\]
conditions on the $28$-dimensional moduli space $\mathcal M_3$.
Consequently, every component of this stratum through our explicit
point has dimension at least $19$.

Every surface in this stratum has Picard number at least $11$, because
the zero section, a fiber, and the nine nonidentity components of the
reducible fibers give eleven independent divisor classes.  On the
other hand, Kloosterman's dimension theorem gives
$\dim(\operatorname{NL}_{11}\cap U)=19$ and
$\dim(\operatorname{NL}_{12}\cap U)=18$; see
\cite[Theorem~1.1]{Klo07}.  Hence the component through
\eqref{eq:odd-Weierstrass} has dimension exactly $19$ and is not
contained in $\operatorname{NL}_{12}$.  Its very general member
therefore has Picard number exactly $11$.

After passing to a finite cover which labels the components of the
reducible fibers, simultaneous resolution gives relative divisors
$D_{\boldsymbol\epsilon}$.  Upper semicontinuity, together with the
existence of the canonical section, shows that
$h^0(X,\OO_X(D_{\boldsymbol\epsilon}))=1$ on a nonempty open subset,
simultaneously for all eight choices of $\boldsymbol\epsilon$.  We can
therefore choose a member of this open set with Picard number $11$.
The Shioda--Tate formula gives Mordell--Weil rank zero.  Since the
Mordell--Weil group is finitely generated, it is finite, and Lemma
\ref{lem:MW-zero} then shows that it is trivial.

By \cite[Propositions~4.1 and~4.3]{LLL26}, the eight isotropic
selections above give precisely the eight nonvertical isotropic
extremal rays of the nef cone.  Proposition
\ref{prop:primitive-generator} shows that the divisors
$D_{\boldsymbol\epsilon}$ are their primitive integral generators.
Since each of them has only its canonical section, Theorem
\ref{thm:normal-obstruction} gives
$\kappa(X,D_{\boldsymbol\epsilon})=0$ and shows that none is
semiample.
Finally, Proposition \ref{prop:isotropic-pushforward} gives rank six
and degree $-6$ for
$\pi_*\OO_X(D_{\boldsymbol\epsilon})$.  Since
$h^0(X,\OO_X(D_{\boldsymbol\epsilon}))=1$, the same proposition shows
that its splitting contains exactly one trivial summand and that all
the remaining summands have negative degree.  Five negative integers
with total degree $-6$ must be $-2,-1,-1,-1,-1$.  Hence
\[
\pi_*\OO_X(D_{\boldsymbol\epsilon})
\simeq
\OO_{\PP^1}\oplus\OO_{\PP^1}(-2)
\oplus\OO_{\PP^1}(-1)^{\oplus4}.
\]
This completes the proof.
\end{proof}
\begin{proposition}
\label{prop:odd-example-bcp}
Let $\pi: X\to \bP^1$ be a semistable Jacobian elliptic surface with $\delta(\pi)=\chi(\mathcal O_X)=3$ and finite Mordell-Weil group.
Then $X$ satisfies the BCP.
In particular, the surface $X$ of Theorem~\ref{thm:odd-counterexample} satisfies the BCP. More precisely, for every reduced and
irreducible curve $C\subset X$,
\[
  h^1(X,\OO_X(C))\leq h^0(X,\OO_X(C)).
\]
Thus one may take $c_X=1$, although none of the nonvertical isotropic
extremal nef rays is semiample and $X$ is not a Mori dream surface.
\end{proposition}
\begin{proof}
We have $q(X)=0$, $K_X\sim F$, and
$\overline{\operatorname{NE}}(X)=\mathcal C_\pi$.
By Lemma~\ref{lem:MW-zero}, the classes of $C_0$, $F$, and the
nonidentity components of the reducible fibers form an integral basis
of $\Pic(X)$.
For every selection $S$, the normalized supporting class
$H_S=C_0+3F-\sum_i\omega_{T_i}$ satisfies
$$
 H_S\cdot C_0=0,\qquad H_S\cdot F=1,\qquad
 H_S\cdot C_{i,j}\in\{0,1\}\quad (j\ne0).
$$
Consequently, although $H_S$ need not be an integral divisor class,
\begin{equation}
\label{eq:HS-integral-pairing}
 H_S\cdot L\in\mathbb Z\qquad\text{for every }L\in\Pic(X).
\end{equation}

To show the BCP for $X$, by \cite[Propositions 2.3 and 3.1]{Li26b}, it suffice to show that there exists a constant $c_X>0$ such that either $h^1(\mathcal O_X(C))\le c_Xh^0(\mathcal O_X(C))$ for every reduced and irreducible curve $C$ with $C^2>0$ on $X$.
Note that $C$ is nef and big. 
By the Hodge index theorem, a big and nef class has positive
intersection with every nonzero nef class. Hence
$C\cdot H_S>0$ for every selection $S$. By
\eqref{eq:HS-integral-pairing}, this gives $C\cdot H_S\geq1$.
Write
\[
 [C]=a[F]+\sum_S b_S[H_S],\qquad a,b_S\geq0.
\]
Using $F^2=0$ and $F\cdot H_S=1$, we obtain
\[
 C^2=a(C\cdot F)+\sum_S b_S(C\cdot H_S)
     \geq\sum_S b_S
     =F\cdot C
     =K_X\cdot C.
\]
Since $(K_X-C)\cdot F=-C\cdot F<0$, Serre duality gives
$h^2(X,\OO_X(C))=0$. Riemann--Roch therefore yields
\[
 h^1(X,\OO_X(C))
 =h^0(X,\OO_X(C))-3+\frac{K_X\cdot C-C^2}{2}
 \leq h^0(X,\OO_X(C))-3.
\]
Thus, $X$ satisfies the BCP.

Now let $X$ be as in Theorem \ref{thm:odd-counterexample}.
Let $C$ be a horizontal reduced and irreducible curve distinct from
$C_0$. It intersects $C_0$ and every component of every reducible
fiber nonnegatively. Since these classes generate the Mori cone,
$C$ is nef.
Moreover, $C^2>0$. Indeed, if $C^2=0$, its class would span one of
the eight nonvertical isotropic extremal nef rays. Since the
corresponding $D_S$ is primitive, we would have $C\sim mD_S$ for
some positive integer $m$. Theorem~\ref{thm:odd-counterexample} and
$\kappa(X,D_S)=0$ imply $h^0(X,\OO_X(mD_S))=1$ for every $m\geq1$.
Thus the sole member of $|mD_S|$ is the reducible effective divisor
$mD_S$, a contradiction.
Thus, it remains to consider $C_0$ and vertical prime curves.
For $C_0$, we have $h^0(\OO_X(C_0))=1$, $h^2(\OO_X(C_0))=0$,
and $\chi(\OO_X(C_0))=1$, so $h^1(\OO_X(C_0))=0$.
An irreducible fiber is linearly equivalent to $F$; thus
$h^0(\OO_X(F))=2$, $h^2(\OO_X(F))=h^0(\OO_X)=1$, and
$\chi(\OO_X(F))=3$, giving $h^1(\OO_X(F))=0$.
Finally, let $T$ be a component of a reducible fiber. Then
$T^2=-2$, $K_X\cdot T=0$, and $h^0(\OO_X(T))=1$.
Moreover,
\[
 h^2(\OO_X(T))=h^0(\OO_X(F-T))=1:
\]
among the divisors in $|F|$, precisely the fiber containing $T$
contains $T$ as a component. Since $\chi(\OO_X(T))=2$, we again
obtain $h^1(\OO_X(T))=0$.
The stated inequality follows in every case.
\end{proof}
\begin{proof}[Proof of Theorem \ref{thm-odd}]
It follows from Theorem \ref{thm:odd-counterexample} and Proposition \ref{prop:odd-example-bcp}.
\end{proof}

\end{document}